\documentclass[11pt]{amsart}

\usepackage{amsmath,amsthm,amssymb,amscd}

\newcommand{\coana}{\mathbf{\Pi}^1_1}

\newcommand{\cl}{\mathrm{cl}}

\newcommand{\cantor}{2^\omega}
\newcommand{\ball}{\mathrm{ball}}
\newcommand{\Fin}{\mathrm{Fin}}
\newcommand{\fsigma}{\mathbf{\Sigma}^0_2}
\newcommand{\gdelta}{\mathbf{\Pi}^0_2}
\newcommand{\fsigmadelta}{\mathbf{\Pi}^0_3}
\newcommand{\Q}{\mathbb{Q}}
\newcommand{\RB}{\mathrm{RB}}
\newcommand{\baire}{\omega^\omega}

\newtheorem{theorem}{Theorem}[section]
\newtheorem{lemma}[theorem]{Lemma}
\newtheorem{corollary}[theorem]{Corollary}
\newtheorem{proposition}[theorem]{Proposition}
\newtheorem{claim}[theorem]{Claim}

\theoremstyle{definition}


\begin{document}

\title{Topological representations}

\author{Adam Kwela} 

\address{Instytut Matematyczny Polskiej Akademii Nauk,
  ul. \'Sniadeckich 8, 00-956 Warszawa, Poland}

\email{A.Kwela@impan.pl}

\author{Marcin Sabok}\thanks{This research was partially
  supported by the MNiSW (the Polish Ministry of Science and
  Higher Education) grant no. N201 418939 and by the
  Foundation for Polish Science.}

\address{Instytut Matematyczny Polskiej Akademii Nauk,
  ul. \'Sniadeckich 8, 00-956 Warszawa, Poland}
\address{Instytut Matematyczny Uniwersytetu Wroc\l awskiego,
  pl. Grunwaldzki 2/4, 50--384 Wroc\l aw, Poland}

\email{M.Sabok@impan.pl}

\date{}

\begin{abstract}
  This paper studies the combinatorics of ideals which
  recently appeared in ergodicity results for analytic
  equivalence relations. The ideals have the following
  topological representation. There is a separable
  metrizable space $X$, a $\sigma$-ideal $I$ on $X$ and a
  dense countable subset $D$ of $X$ such that the ideal
  consists of those subsets of $D$ whose closure belongs to
  $I$. It turns out that this definition is indepedent of
  the choice of $D$. We show that an ideal is of this form
  if and only if it is dense and countably separated. The
  latter is a variation of a notion introduced by Todor\v
  cevi\'c for gaps. As a corollary, we get that this class
  is invariant under the Rudin--Blass equivalence. This also
  implies that the space $X$ can be always chosen to be
  compact so that $I$ is a $\sigma$-ideal of compact
  sets. We compute the possible descriptive complexities of
  such ideals and conclude that all analytic equivalence
  relations induced by such ideals are $\fsigmadelta$. We
  also prove that a coanalytic ideal is an intersection of
  ideals of this form if and only if it is weakly selective.
\end{abstract}


\maketitle

\section{Introduction}

The aim of this paper is to reveal a connection between the
structure of ideals on countable sets and ideals of compact
sets in Polish spaces. A family of subsets of a given set is
an \textit{ideal} if it is closed under taking subsets and
finite unions. We always assume that an ideal of subsets of
a set $S$ contains all singletons $\{s\}$ for $s\in S$,
i.e. an ideal $J$ is an ideal of subsets of $\bigcup
J$. Given an ideal $J$, we say that a set is
\textit{$J$-positive} if it does not belong to
$J$. Sometimes, we write $J^+$ for the family of
$J$-positive sets and co-$J$ for the dual filter. Throughout
this paper we often identify subsets of $\omega$ with
elements of $\cantor$ via the characteristic
functions. Thus, for example, given an ideal of subsets of
$\omega$ we define its descriptive complexity as if it were
a subset of $\cantor$. On the other hand, given an ideal of
compact subsets of a given Polish space $X$ we refer to its
descriptive complexity in the Vietoris space $K(X)$.

The study of definable ideals of compact sets has become a
classical subject in descriptive set theory. A well-known
result of Kechris, Louveau and Woodin \cite{klw} and
Dourgherty, Kechris and Louveau (see \cite{k2}) says that an
analytic ideal of compact sets is a $\sigma$-ideal if and
only if it is $\gdelta$. The descriptive complexity of more
complicated ideals of compact sets is the subject of a
trichotomy theorem proved by Matheron, Solecki and Zelen\'y
\cite{msz}. Recently, Solecki \cite{sol.jems} described a
special class of $\gdelta$ $\sigma$-ideals of compact sets
and proved several structure theorems representing ideals in
that class via the meager ideal (see also
\cite{kechris.solecki}).

The structure of ideals on countable sets has been studied
from a different perspective but some results reveal the
similarities. An analogy to the Kechris--Louveau--Woodin
theorem appears in the work of Solecki \cite{solecki.p} on
analytic $P$-ideals, where it is shown that if a $P$-ideal
is analytic, then its descriptive complexity is
$\fsigmadelta$. Solecki also shows \cite[Corollary
3.4]{solecki.bull} that if $J$ is an analytic $P$-ideal then
it is either $\fsigma$ or $\fsigmadelta$-complete. The
structure of ideals on $\omega$ is often described in terms
of the \textit{Rudin--Blass order}. Given two ideals $J,K$
on $\omega$ we say that $J$ is \textit{Rudin--Blass below}
$K$ and write $J\leq_{\RB}K$ if there is a finite-to-one
function $f:\omega\rightarrow\omega$ such that $a\in J$ if
and only if $f^{-1}(a)\in K$, for every
$a\subseteq\omega$. The Jalali-Naini--Mathias--Talangrand
theorem \cite[Theorem 4.1.2]{bartoszynski} then says that
every ideal with the Baire property is Rudin--Blass above
the ideal $\Fin$ of finite sets.

A connection between ideals of compact sets and ideals on
countable sets that appears in this paper uses the following
operation. Suppose $X$ is a separable metrizable space and
$I$ is a $\sigma$-ideal on $X$ that contains all
singletons. Given a dense countable set $D\subseteq X$,
define the ideal $J_I$ on $D$ as the family $\{a\subseteq D:
\cl(a)\in I\}$. Obviously, $J_I$ depends only on the family
of closed sets that belong to $I$. In principle, $J_I$ also
depends on the set $D$, which is equal to $\bigcup J_I$, but
we will see (Proposition \ref{independence}) that, up to
isomorphism, this definition is independent of the choice of
$D$. The ideals of the form $J_I$ have been recently studied
in \cite{sz} and used in canonization (see \cite{ksz}) of
smooth equivalence relations for $\sigma$-ideals generated
by closed sets. Given an ideal $J$ on a countable set $E$ we
say that $J$ \textit{has a topological representation} if
there are $I,D,X$ as above and a bijection
$\rho:E\rightarrow D$ such that $J=\{a\subseteq E:
\rho(a)\in J_I\}$. In such a case we say that $J$ is
\textit{represented by} $I$.

Two examples of ideals with topological representations have
been studied by Farah and Solecki \cite{farah.solecki}, who
showed that there are at least two isomorphism types of such
ideals (namely that the ideals represented by the meager
sets and by the meager null sets are not isomorphic).

The study of ideals on $\omega$ is closely connected and
largely motivated by the study of equivalence relations on
$\cantor$ given in the following way. Given an ideal $J$ on
$\omega$ we write $E_J$ for the equivalence relation on
$\cantor$ with $x\mathrel{E_J}y$ if $x\Delta y\in
J$. Rosendal \cite{rosendal} proved that any Borel
equivalence relation is Borel-reducible to one of the form
$E_J$. A motivation for the results in this paper is a
recent work of Zapletal \cite{zapletal}, who shows that if
$J$ has a topological representation, then the equivalence
relation $E_J$ has the following ergodicity property. First,
every Borel homomorphism from a turbulent equivalence
relation $F$ to $E_J$ maps a comeager set to a single
$E_J$-equivalence class. Second, if $J$ is represented by a
$\mathbf{\Pi}^0_2$ $\sigma$-ideal, then every homomorphism
from $E_J$ to an equivalence relation classifiable by
countable structures maps a measure 1 set to a single
equivalence class. This is in contrast with the turbulence
dichotomy of Hjorth \cite{hjorth}. Note that if $J$ is has a
topological representation, then it is not a $P$-ideal and
hence $E_J$ is not induced by a Polish group action by the
Solecki characterization of Polishable $P$-ideals
\cite[Corollary 4.1]{solecki.p}.

The main result of this paper establishes a combinatorial
characterization of ideals which admit topological
representations. One of the necessary conditions says that
the ideal is \textit{dense}, i.e. any infinite set contains
an infinite subset that belongs to the ideal. The other
condition is a variation of Todor\v cevi\'c's notion of
countably separated gaps \cite{todorcevic.gaps} (see also
\cite{farah.mems, farah.hausdorff, todorcevic.aviles}). We
say that an ideal $J$ on a countable set $D$ is
\textit{countably separated} if there is a countable family
$\{a_n:n<\omega\}$ of subsets of $D$ such that for any
$a,b\subseteq D$ with $a\in J$ and $b\notin J$ there is
$n\in\omega$ with $a\cap a_n=\emptyset$ and $b\cap a_n\notin
J$. In such a case we say that the family $\{a_n:n<\omega\}$
\textit{separates} $J$. We prove the following
characterization.

\begin{theorem}\label{char} 
  For any ideal $J$ on a countable set the following are
  equivalent:
  \begin{itemize}
  \item[(i)] $J$ is dense and countably separated,
  \item[(ii)] $J$ has a topological representation.
  \end{itemize}
\end{theorem}

As a corollary we get the following

\begin{corollary}\label{rudin-blass}
  The class of ideals which have topological representations
  is invariant under the Rudin--Blass equivalence.
\end{corollary}

\begin{proof}
  It is enough to show that if $J\leq_\RB K$ and $K$ is
  countably separated, then $J$ is countably separated, and
  if $J$ is dense, then $K$ is dense. Let
  $f:\omega\rightarrow\omega$ be a Rudin--Blass reduction
  witnessing $J\leq_\RB K$.

  Suppose first that $K$ is countably separated by
  $\{c_n:n<\omega\}$ and let $d_n=f''c_n$. We claim that
  $d_n$'s witness that $J$ is countably separated. Indeed,
  take $a\in J$ and $b\notin J$. Then $a'=f^{-1}(a)\in K$
  and $b'=f^{-1}(b)\notin K$. Pick $n$ such that $c_n\cap
  a'=\emptyset$ and $c_n\cap b'\notin K$. Then $d_n\cap
  a=\emptyset$ and $d_n\cap b\notin K$.

  Suppose now that $J$ is dense and let $b\subseteq\omega$
  be infinite. Since $f$ is finite-to-one, $b'=f''b$ is also
  infinite and hence there is $c'\subseteq b'$ with $c'\in
  J$. Let $c=f^{-1}(c')\cap b$ and note that $c$ is an
  infinite subset of $b$ which belongs to $K$, as a subset
  of $f^{-1}(c')$.
\end{proof}

Corollary \ref{rudin-blass} in particular implies that
ideals with topological representations are invariant under
$\equiv_\RB^{++}$ (see \cite[Section 3.2]{kanovei}) and
hence the class of equivalence relations of the form $E_J$,
for $J$ with a topological representation, is invariant
under additive Borel reductions, by a result of Farah
\cite{farah.additive}.

The proof of Theorem \ref{char} shows that if $J$ on
$\omega$ has a topological representation, then it can be
represented on the Cantor space via an identification of
$\omega$ with the set of rationals in the Cantor space. The
following shows that there is some extent of control over
the sets in the $\sigma$-ideal that witness that $J$ is
represented on the Cantor space.

\begin{theorem}\label{meager}
  If $J$ has a topological representation, then it is
  represented on the Cantor space by a $\sigma$-ideal
  generated by a family of compact nowhere dense sets.
\end{theorem}

In partiular, the family representing an ideal can be chosen
to consist of compact sets. For definable ideals, this
implies computations of possible descriptive complexities.

\begin{theorem}\label{complex}
  If an ideal $J$ has a topological representation and it is
  analytic, then it is $\mathbf{\Pi}^0_3$-complete. In
  particular, if $E_J$ is analytic, then it is
  $\mathbf{\Pi}^0_3$.
\end{theorem}

This gives an analogoue of the Kechris--Louveau--Woodin
dichotomy.

\begin{corollary}\label{klw}
  If a coanalytic ideal $J$ has a topological
  representation, then $J$ is either
  $\mathbf{\Pi}^0_3$-complete, or
  $\mathbf{\Pi}^1_1$-complete. $J$ is
  $\mathbf{\Pi}^0_3$-complete if and only if it is
  represented by a $\mathbf{\Pi}^0_2$ ideal of compact sets.
\end{corollary}

Besides the descriptive complexity, there is one more
combinatorial condition that determines the structure of
ideals which have topological representations. An ideal $J$
is \textit{weakly selective} if for every $b\notin J$ and a
function $f:b\rightarrow \omega$ there is a $J$-positive
subset $a$ of $b$ such that $f\restriction a$ is either
one-to-one or constant. Equivalently, $J$ is weakly
selective if any partition of a $J$-positive set into sets
in $J$ admits a $J$-positive selector. Weakly selective
ideals have been studied by several authors (see Farah
\cite{farah.semiselective} or Baumgartner and Laver
\cite{baumgartner.laver}) and \cite[Proposition 4.3]{sz}
shows that if $J$ has a topological representation, then it
is weakly selective. Here we prove the following
characterization.

\begin{theorem}\label{ws}
  Let $J$ be a coanalytic ideal. The following are
  equivalent:
  \begin{itemize}
  \item[(i)] $J$ is weakly selective,
  \item[(ii)] $J$ is an intersection of a family of ideals
    with topological representations.
  \end{itemize}
\end{theorem}

\bigskip

The paper is organized as follows. A preliminary discussion
and a proof of Theorem \ref{char} are given in Section
\ref{sec:char}. Theorem \ref{meager} is proved in Section
\ref{sec:meager}. Theorem \ref{complex} together with
Corollary \ref{klw} are proved in Section
\ref{sec:complex}. Section \ref{sec:ws} contains a proof of
Theorem \ref{ws}.

\section{A characterization of ideals with topological representations}\label{sec:char}

The definition on the ideals $J_I$ formally depends on the
choice of the dense set $D$. It turns out, however, that no
matter what dense set $D$ is chosen, the ideal $J_I$ is the
same, up to isomorphism.

\begin{proposition}\label{independence}
  Given a $\sigma$-ideal $I$ on a separable metric space $X$
  and two dense countable sets $D$ and $E$ in $X$, if
  $J=\{a\subseteq D: \cl(a)\in I\}$ and $K=\{a\subseteq E:
  \cl(a)\in I\}$, then $J$ and $K$ are isomorphic.
\end{proposition}
\begin{proof}
  Using a back-and-forth argument, enumerate
  $D=\{d_n:n<\omega\}$ and $E=\{e_n:n<\omega\}$ so that the
  distance of $d_n$ and $e_n$ is smaller than $1\slash
  n$. Now for $a\subseteq\omega$ write
  $\cl_D(a)=\cl(\{d_n:n\in a\})$ and
  $\cl_E(a)=\cl(\{e_n:n\in a\})$, where $\cl$ stands for the
  closure taken in $X$. To see that $J$ and $K$ are
  isomorphic, it is enough to show that $\cl_D(a)$ belongs
  to $I$ if and only if $\cl_E(a)$ belongs to $I$. Note that
  \begin{enumerate}
  \item $\cl_D(a)\subseteq\cl_E(a)\cup\{d_n:n\in a\}$,
  \item $\cl_E(a)\subseteq\cl_D(a)\cup\{e_n:n\in a\}$.
  \end{enumerate}
  Here, (1) follows from the fact that if $x$ belongs to
  $\cl_D(a)$ and is not one of the $d_n$'s (for $n\in a$),
  then there is an infinite subsequence of $d_n$'s (indexed
  with elements of $a$) that converges to $x$. Since $e_n$
  is $(1\slash n)$-close to $d_n$, there is also an infinite
  subsequence of $e_n$'s (with the same index set)
  converging to $x$. (2) follows by symmetry. Now, (1) and
  (2) imply that $\cl_D(a)$ and $\cl_E(a)$ can differ by an
  at most countable set. Since the singletons belong to $I$
  (we always assume that ideals contain all singletons), it
  follows that $\cl_D(a)\in I$ if and only if $\cl_E(a)\in
  I$.
\end{proof}

Now we will prove Theorem \ref{char}. Let us first comment
on the sharpness of condition (i) in that theorem: neither
being dense nor countably separated alone implies that the
ideal has a topological representation.

To see that, first consider the ideal
$\emptyset\times\Fin=\{a\subseteq\omega\times\omega:\forall
n\in\omega\ a_n\in\Fin\}$, where $a_n=\{m\in\omega:(n,m)\in
a\}$. This ideal is countably separated, by the sets
$c_{n,k}=\{(n,m)\in\omega\times\omega:m>k\}$ but it is
clearly not dense.

On the other hand, consider the ideal
$\Fin\times\Fin=\{a\subseteq\omega\times\omega:
\{n\in\omega:a_n\in\Fin\}\in\Fin\}$. The ideal
$\Fin\times\Fin$ is not weakly selective, as witnessed by
the projection function $(n,m)\mapsto n$. By
\cite[Proposition 4.3]{sz}, an ideal which has a topological
representation is weakly selective. Thus, $\Fin\times\Fin$
does not have a topological representation. On the other
hand, it is clearly dense. Thefore, by Theorem \ref{char} it
cannot be countably separated.

Below, for $\varepsilon>0$ and $A\subseteq\cantor$, write
$$\ball(\varepsilon,A)=\{x\in\cantor: \exists y\in A\quad
d(x,y)<\varepsilon\}.$$

\begin{proof}[Proof of Theorem \ref{char}]
  (i)$\Rightarrow$(ii) Suppose $J$ is represented on $X$ and
  let $D\subseteq X$ be countable dense, $I$ be a
  $\sigma$-ideal such that $J=J_I$. First note that $J_I$ is
  dense. Indeed, take an infinite $a\notin J_I$. Then
  $\cl(a)\notin I$, so in particular $\cl(a)$ is
  uncountable. Let $x\in\cl(a)\setminus a$ and pick a
  sequence $\langle x_n: n\in\omega\rangle $ of elements of
  $a$ converging to $x$. Then $b=\{x_n:n<\omega\}$ is an
  infinite subset of $a$, which is in $J_I$ since
  $\cl(b)=b\cup\{x\}$ is countable and hence in $I$.

  To see that $J_I$ is countably separated, fix a countable
  basis $\{U_n:n<\omega\}$ of $X$ and let $c_n=U_n\cap
  D$. We claim that $\{c_n:n<\omega\}$ witnesses that $J$ is
  countably separated. Indeed, let $a,b\subseteq D$ be such
  that $a\notin J_I,b\in J_I$. Then $\cl(a)\notin I$ and
  $\cl(b)\in I$. By countable additivity of $I$, there
  exists $n$ such that $U_n\cap\cl(b)=\emptyset$ and
  $U_n\cap\cl(a)\notin I$. Then clearly $c_n\cap
  b=\emptyset$ and $c_n\cap a$ is $J$-positive since
  $\cl(c_n\cap a)$ contains $U_n\cap\cl(a)$.

  (ii)$\Rightarrow$(i) Suppose now that $J$ is countably
  separated and dense. Assume that $J$ is an ideal on
  $\omega$. We will first show that the family
  $\{c_n:n\in\omega\}$ witnessing that $J$ is countably
  separated can be improved a little. We say that a family
  of subsets of $\omega$ \textit{separates points} if for
  each $n\not=m\in\omega$ there is a set $a$ in that family
  such that $n\in a$ and $m\notin a$.

\begin{lemma}\label{improved}
  If $J$ is countably separated, then there is a family
  witnessing that $J$ is countably separated, which
  separates points and is such that all Boolean combinations
  of its elements are either infinite or empty.
\end{lemma}

\begin{proof}
  Let $\{c_n:n\in\omega\} $ be a family witnessing that $J$
  is countably separated. Enumerate all pairs of distinct
  natural numbers as
  $\langle(k_n,l_n):n\in\omega\rangle$. We will construct a
  family $\{d_n:n\in\omega\}$ of subsets of $\omega$ such
  that for each $n$ the following is true
  \begin{itemize}
  \item[(a)] if $n=2m$ is even, then $d_n$ is a subset of $c_m$
    such that $c_m\setminus d_n\in J$;
  \item[(b)] if $n=2m+1$ is odd, then $k_m\in d_n$ and $l_n\notin
    d_n$;
  \item[(c)] all Boolean combinations of $d_i$ for $i\leq n$ are
    infinite or empty.
  \end{itemize}

  Notice that such a family will also witness that $J$ is
  countably separated by (a). It will separate points by (b)
  and have all Boolean combinations either empty of infinite
  by (c). Hence, $\{d_n:n<\omega\}$ will be the required
  family.

  To construct the sets $d_n$ inductively, we start with
  $d_0=c_0$. Suppose that $d_k$ for $k<n$ have been
  constructed. All Boolean combination of $\{d_k: k<n\}$
  define a finite partition $\{a_k:k<k_n\}$ of $\omega$ into
  infinite subsets.

  \textbf{Case 1}.  Suppose that $n=2m$ is even. For each
  $k<k_n$ we define a set $e_k\subset a_k\cap c_m$ in the
  following way. There are three possibilities:
  \begin{itemize}
  \item if $a_k\cap c_m\in J$, then $e_k=\emptyset$;
  \item if $a_k\cap c_m\notin J$ and $a_k\setminus c_m$ is
    infinite, then $e_k=a_k\cap c_m$;
  \item if $a_k\cap c_m\notin J$ and $a_k\setminus c_m$ is
    finite, then find an infinite subset $e_k'\in J$ of
    $a_k\cap c_m$ (using the fact that $J$ is dense) and
    define $e_k=\left(a_k\cap c_m\right)\setminus e_k'$.
 \end{itemize}
 The set $d_n=\bigcup_{k<k_n}e_k$ is a subset of $c_m$ such
 that $c_m\setminus d_n\in J$. Also, $d_n$ is either empty
 or both infinite and coinfinite in every $a_k$, therefore
 it is as needed.

 \textbf{Case 2}. Suppose that $n=2m+1$ is odd. There is
 $k<k_n$ such that $k_m\in a_k$. Let $d_n$ be any infinite
 subset of $a_k$ such that $k_m\in d_n$, $l_m\notin d_n$ and
 $a_k\setminus d_n$ is infinite. Then $d_n$ separates the
 pair $(k_m,l_m)$ and in each $a_k$ it is either empty or
 infinite and coinfinite, therefore it is as needed.
\end{proof}

We can now assume that a family $\{c_n:n<\omega\}$
witnessing that $J$ is countably separated is as in Lemma
\ref{improved}. Define a topology $\tau$ on $\omega$ by
letting all $c_n$'s be clopen basic sets. This is a
Hausdorff, second-countable and regular topology, since it
is zero-dimensional. By Urysohn's metrization theorem
\cite[Theorem 1.1]{kechris}, it is metrizable. Note that
since all Boolean combinations of the elements of the basis
are either empty or infinite, this space has no isolated
points. Then, as a countable metrizable topological space
without isolated points, is homeomorphic to the rationals,
by a theorem of Sierpi\'nski \cite[Exercise
7.12]{kechris}. Embed $\left(\omega,\tau\right)$ into the
Cantor set $2^\omega$ so that it is homeomorphic to
$D=\mathbb{Q}\cap2^\omega$. Thus, via this embedding, we
treat now $J$ as an ideal on $D$.

Define an ideal of $K(\cantor)$ by $I=\{A\in
K(\cantor):\exists a\in J\quad A\subset\cl(a)\}$. It turns
out that $I$ is a $\sigma$-ideal on $K(\cantor)$.

\begin{lemma}
  $I$ is a $\sigma$-ideal of compact sets.
\end{lemma}

\begin{proof}
  Suppose it is not a $\sigma$-ideal. In that case there is
  a sequence of compact sets $A_n\in I$ such that their
  union $A=\bigcup_{n<\omega}A_n$ is also compact and does
  not belong to $I$. Without loss of generality we can
  assume that $A_n$'s are increasing. Fix a metric $d$ on
  $\cantor$ of diameter $\leq 1$. The metric notions below
  refer to the metric $d$.

  We claim that there is $n\in\omega$ such that $A\setminus
  \ball(\varepsilon,A_n)\in I$ for every
  $\varepsilon>0$. Suppose otherwise and construct an
  increasing sequence of natural numbers $n_i$ and a
  sequence of reals $\varepsilon_i>0$ such that:
  \begin{itemize}
  \item $A\setminus \ball(\varepsilon_i,A_{n_i})$ does not
    belong to $I$,
  \item $A_{n_{i+1}}$ is not contained in
    $\ball(\varepsilon_i,A_{n_i})$.
  \end{itemize}
  This is easy to do using our assumption and the fact that
  $A_n$'s exhaust $A$. But then
  $A\setminus\bigcup_{i<\omega}
  \ball(\varepsilon_i,A_{n_i})$ is nonempty, by compactness
  of $A$. On the other hand, if $x\in
  A\setminus\bigcup_{i<\omega}
  \ball(\varepsilon_i,A_{n_i})$, then $x\in
  A\setminus\bigcup_{n<\omega}A_n$, which gives a
  contradiction.

  Fix a number $n$ as in the previous paragraph and without
  loss of generality assume that $n=0$. Let
  $B_0=A\setminus\ball(1,A_0)$ and for each $k\geq 1$ let
  $B_k=A\cap(\ball(\frac{1}{k},A_0)\setminus
  \ball(\frac{1}{k+1},A_0))$. Note that each $B_k$ is in
  $I$, by our assumption. Next, for each $k\in\omega$ find a
  set $b_k\subseteq D$ such that $b_k\in J$,
  $B_k\subseteq\cl(b_k)$ and
  $b_k\subseteq\ball(\frac{1}{k},B_k)$. Find also $a\in J$
  such that $A_0\subseteq\cl(a)$. Let
  $b=\bigcup_{k<\omega}b_k\cup a$. Note that
  $A\subseteq\cl(b)$, so $b\notin J$. Since $J$ is countably
  separated by $c_n$'s, there is $n$ such that $a\cap
  c_n=\emptyset$ and $b\cap c_n\notin J$. Now, since $c_n$'s
  are clopen on $D$, we get a clopen set $C\subseteq\cantor$
  such that $a\subseteq C$ and $C\cap c_n=\emptyset$. Let
  $\varepsilon>0$ be such that
  $\ball(\varepsilon,A_0)\subseteq C$. By the definition of
  $b_k$'s, all but finitely many of them are contained in
  $\ball(\varepsilon,A_0)$. Hence $b\setminus C$ is covered
  with finitely many of the sets $b_k$, and so is $b\cap
  c_n\subseteq b\setminus C$. Since each $b_k$ belongs to
  $J$, this contradicts the fact that $b\cap c_n\notin J$.
\end{proof}

Now, to finish the proof we will show that $J=J_I$. One
inclusion is obvious: if $a\in J$, then $\cl(a)\in I$ by the
definition of $I$ and so $a\in J_I$. On the other hand, if
$a\in J_I$, then $\cl(a)\in I$. Thus, there is $b\in J$ with
$\cl(a)\subseteq\cl(b)$. We must prove that $a\in
J$. However, if $a\notin J$, then for some $n$ we have
$c_n\cap b=\emptyset$ and $c_n\cap a\not=\emptyset$. Let
$C\subseteq\cantor$ be a basic clopen set with $C\cap D=c_n$
and note that $\cl(b)\cap C=\emptyset$ and $\cl(a)\setminus
C\not=\emptyset$, which contradicts
$\cl(a)\subseteq\cl(b)$. Thus, it must be the case that
$a\in J$, which concludes the argument that $J=J_I$ and ends
the entire proof.

\end{proof}

\section{Representation via compact nowhere dense sets}
\label{sec:meager}

It is fairly easy to see that if $J$ has a topological
representation, then it is also represented on a compact
metric space. Indeed, if $J$ is represented on $X$ via $I$,
then let $\hat X$ be a metric compactification of $X$ and
let $\hat I$ be the $\sigma$-ideal on $\hat X$ generated by
the sets $\cl(K)$ for $K\in I$ (the closure is taken in
$\hat X$) and the singletons $\{x\}$ for $x\in\hat
X\setminus X$. Then $J_I$ is represented on $\hat X$ via
$\hat I$ as witnessed by the same dense set $D\subseteq\hat
X$. The proof of Theorem \ref{char} shows something more: if
$J$ has a topological representation, then we can actually
find a $\sigma$-ideal of $I$ closed subsets of the Cantor
space $\cantor$ such that $J$ is represented by $I$. In the
proof of Theorem \ref{meager} we will use similar arguments
and we will make sure that all closed sets in $I$ are
nowhere dense.

\begin{proof}[Proof of Theorem \ref{meager}]
  Suppose $J$ is represented on $X$ via a $\sigma$-ideal
  $I$. By the remarks above, we can assume $X$ is the Cantor
  space and $J$ is a family of subsets of
  $D=\Q\subseteq\cantor$. Note that in this case a family
  witnessing that $J$ is countably separated can be chosen
  to consist of those basic clopen subsets of $\cantor$
  which are $J$-positive. Enumerate the $J$-positive basic
  clopen subsets of $D$ as $\{c_n:n<\omega\}$. Below, the
  notions of open and clopen will refer to the topology on
  $D$.

  \begin{lemma}\label{separation}
    For each open $J$-positive $a\subseteq D$ and distinct
    $k,l\in a$ there are disjoint $J$-positive clopen subsets
    $b,c\subseteq a$ such that $k\in b$ and $l\in c$.
  \end{lemma}
  \begin{proof}
    Let $W_k,W_l$ be disjoint clopen neighborhoods of $k$
    and $l$ in $a$ such that $a\setminus (W_k\cup W_l)$ is
    $J$-positive. Note that such neighborhoods must exist,
    since otherwise $\cl(a)$ would be covered by
    $$\bigcup\{\cl(a\setminus (W_k\cup W_l): W_k,W_l\mbox{
      clopen neighborhoods of }k,l\}\cup\{k,l\}.$$ Let
    $a'=a\setminus (W_k\cup W_l)$. It is enough to show that
    there are two disjoint $J$-positive clopen subsets of
    $a'$. This is to say that $J$ is not a maximal ideal
    below $a'$. Write $A=\cl(a')$ and let
    $A'=A\setminus\bigcup\{U:U\mbox{ basic open and } U\cap
    A\in I\}$. Obviously, $A'\notin I$ and pick $x\in
    A'$. Again, note that here must be a basic clopen
    neighboorhood $V$ of $x$ such that $A\setminus V\notin
    I$ since otherwise $A$ would be covered by
    $\bigcup\{A\setminus V: V\mbox{ basic clopen
      neighboorhood of }x\}\cup\{x\}$ and belong to
    $I$. Pick such $V$ and let $b=a'\cap V$ and
    $c=a'\setminus V$. Now $b$ and $c$ are disjoint
    $J$-positive subsets of $a'$.
  \end{proof}

  \begin{lemma}\label{openclopen}
    Suppose $c\subseteq D$ is a $J$-positive clopen set and
    $b\subseteq c$ is open such that $c\setminus b\in J$. If
    $b_n\subseteq b$ are $J$-positive clopen sets with
    $b=\bigcup_n b_n$, then for every $J$-positive set
    $d\subseteq b$ there is $n$ with $d\cap b_n\notin
    J$.
  \end{lemma}
  \begin{proof}
    Suppose that $d\subseteq b$ is $J$-positive. We need to
    show that $b_n\cap d\notin J$ for some $n$. Suppose
    otherwise. Since $J$ is weakly selective
    \cite[Proposition 4.3]{sz}, there is $J$-positive
    $e\subseteq d$ such that $|e\cap b_n|\leq 1$ for each
    $n$. This means that $e\cap b$ is discrete and hence
    $\cl(e)\subseteq e\cup \cl(c\setminus b)$ since $c$ was
    a clopen set. This implies that $\cl(e)\in I$ and
    contradicts the fact that $e\notin J$.
  \end{proof}

  We will construct a Hausdorff, zero-dimensional topology
  $\tau$ on $D$ which has no isolated points and such that:
  \begin{itemize}
  \item[(i)] all sets in $J$ are nowhere dense in $\tau$,
  \item[(ii)] for each $b\notin J$ and $a\in J$ there is a
    $\tau$-clopen set $U\subseteq D$ such that $U\cap
    a=\emptyset$ and $U\cap b\notin J$.
  \end{itemize}

  Given that, as in the proof of Theorem \ref{char}, find a
  homeomorphism of $(\omega,\tau)$ and $\Q$ and embed it
  into the Cantor set. Let $I$ be the $\sigma$-ideal
  generated by the sets $\cl_\tau(a)$ for $a\in J$ (here
  $\cl_\tau$ stands for the closure taken in the Cantor set
  in which $(\omega,\tau)$ is embedded).

  The condition (i) implies that the elements of $I$ are
  nowhere dense in the Cantor set and we need to show that
  $J=J_I$. Indeed, if $a\in J$, then $\cl_\tau(a)\in I$ and
  so $a\in J_I$. What is left to prove is that if $b\notin
  J$ and $a_n\in J$, then $\cl_\tau(b)\not\subseteq\bigcup_n
  \cl_\tau(a_n)$. By induction construct a decreasing
  sequence of $\tau$-clopen sets $U_n$ with $a_n\cap
  U_n=\emptyset$ and $U_n\cap b\notin J$ and
  $\mathrm{diam}(U_n)<1\slash n$ (the diameter is computed
  with respect to the usual metric on the Cantor set in
  which $(\omega,\tau)$ is embedded). Having $U_n$
  constructed, let $b_n=U_n\cap b$. Using (ii), find a
  $\tau$-clopen set $U_{n+1}$ such that $U_{n+1}\cap
  a_n=\emptyset$ and $U_{n+1}\cap b_n\notin J$. If needed,
  shrink it so that $\mathrm{diam}(U_{n+1})<1\slash(n+1)$ and still
  $U_{n+1}\cap b_n\notin J$ (this is possible as $U_{n+1}$
  is covered with finitely many relatively clopen sets of
  diameter less than $1\slash(n+1)$). At the end, let
  $x\in\cantor$ belong to $\bigcap_n U_n$. Then
  $x\in\cl_\tau(b)\setminus\bigcup\cl_\tau(a_n)$. This shows
  that if $b\notin J$, then $\cl_\tau(b)\notin I$ and thus
  proves that $J=J_I$.

  To construct the topology $\tau$ we will construct sets
  $a_n$ such that $\{a_n:n<\omega\}$ separates points in $D$
  and (ii), (iii) and (iv) hold, where
  \begin{itemize}
  \item[(iii)] each $a_n$ is clopen and all elements in the
    algebra generated by $\{a_i:i<n\}$ are either empty or
    $J$-positive,
  \item[(iv)] for each $n,m$ and $J$-positive set
    $b\subseteq a_n\cap c_m$ there is $k>n$ such that
    $a_k\subseteq a_n\cap c_m $ and $a_k\cap b\notin J$.
  \end{itemize}

  Having the sets $a_n$ constructed, take them as a clopen
  basis of the topology $\tau$, which is then Hausdorff,
  zero-dimensional and has no isolated points by (iii). To
  see (i) note that if $a_n$ is $\tau$-clopen and $b\in J$,
  then there is $m$ such that $a_n\cap c_m$ is disjoint from
  $b$ since $c_m$'s separate $J$. Then, by (iv) applied to
  $b=a_n\cap c_m$ there is $k$ with $a_k\subseteq a_n\cap
  c_m$ and in particular $a_k$ is disjoint from $b$. This
  shows that $b$ is nowhere dense in $\tau$.

  The construction of the sets $a_n$ will be by induction
  with $a_0=D$. In the construction we will make sure that
  (iii) and (v) hold where
  \begin{itemize}
  \item[(v)] for each $n,m$ there is a sequence $k_i$ such
    that $\bigcup_i a_{k_i}\subseteq a_n\cap c_m$ and
    $(a_n\cap c_m)\setminus\bigcup_i a_{k_i}\in J$ and for
    each $b\subseteq a_n\cap c_m$ with $b\notin J$ there is
    $i$ such that $b\cap a_{k_i}\notin J$.
  \end{itemize}

  Note that (v) implies that (ii) and (iv) hold. Indeed,
  (iv) follows from (v) immmediately. To see (ii), take
  $b\notin J$ and $a\in J$ and let $c_m$ be such that
  $c_m\cap a=\emptyset$ and $c_m\cap b\notin J$. Apply (v)
  to $a_0=D$ and $c_m$ and find $k_i$ such that
  $a_{k_i}\subseteq c_m$ and $b\cap a_{k_i}\notin J$. This
  proves (ii).

  Now we are ready for the induction that takes care of
  (iii) and (v). Start with $a_0=D$. Given $k$ write $A_k$
  for the algebra generated by $\{a_i:i<k\}$. At each step
  of the induction we will make a sequence of promises. A
  \textit{promise} is a pair $(j,c)$, where $j\in D$ and
  $c\subseteq D$ is clopen and $J$-positive. The meaning of
  a promise is as follows. Given a set $a$ in the algebra
  generated by the sets constructed so far and its subset
  $c\subseteq a$ which is clopen and such that $a\setminus
  c\in J$ we will make sure that $c$ is $\tau$-open and
  promise to construct a sequence of $\tau$-clopen and
  clopen sets $a_{k_i}$ in the future so that
  $c=\bigcup_{i}a_{k_i}$. To do so, we must construct one
  $a_{k_i}\subseteq c$ for each $j\in c$ with $j\in
  a_{k_i}$. Thus, the set of promises made in such case is
  the set $\{(j,c):j\in c\}$. The inductive construction
  will use bookkeeping in order to fulfill all promises made
  during its steps.

  Enumerate also all pairs of distinct elements of $D$ as
  $(k_n,l_n)$. At each step $k$ we will consider one of the
  three possibilities:
  \begin{itemize}
  \item[(a)] either we construct $a_k$ to separate $k_n$ and $l_n$,
  \item[(b)] or we consider a pair $a_i, c_m$ for some $i<n$
    and make sure (possibly making promises) that there will
    be a sequence $a_{k_n}$ such that that $\bigcup_n
    a_{k_n}\subseteq c_m\cap a_i$ and for each $J$-positive
    $b\subseteq a_i\cap c_m$ there is $j$ with $b\cap
    a_{k_j}\notin J$,
  \item[(c)] or we consider a promise $(j,c)$ made so far
    and construct $a_k$ so that $j\in a_k$ and $a_k\subseteq
    c$.
  \end{itemize}
  We start with $a_0=D$. Suppose everything is constructed
  so far and we are at step $k$. There are three cases.

  \textbf{Case (a)}. We need to separate $k_n$ and $l_n$. If
  $k_n$ and $l_n$ are already separated by the algebra
  $A_k$, then put $a_k=\emptyset$. Otherwise, find an atom
  $d$ of this algebra with $k_n,l_n\in d$. Use Lemma
  \ref{separation} to find two disjoint clopen $J$-positive
  subsets $U,V\subseteq d$ with $k_n\in U$ and $l_n\in
  V$. Put $a_k=U$.

  \textbf{Case (b)}. Suppose we consider the pair
  $a_i,c_m$. Enumerate all the atoms of the algebra $A_k$
  below $a_i$ as $\{d_l:l<L\}$. Let $c'$ be obtained by
  removing from $c_m$ all the intersections $c_m\cap d_l$
  which are in $J$. Note that $c'$ is still clopen and
  $c\setminus c'\in J$. We will make sure to construct a
  sequence $a_{k_j}$ as in (iii) so that $c'=\bigcup_j
  a_{k_j}$ and for each $b\subseteq c'\cap a_i$ there is $j$
  with $a_{k_j}\cap b\notin J$. Define $P,Q$ to be subsets of the
  set of atoms of $A_k$ with
  \begin{itemize}
  \item $d\in P$ if $c'\cap d\notin J$ and $d\setminus
    c'\notin J$,
  \item $d\in Q$ if $c'\cap d$ is co-$J$ in $d$.
  \end{itemize}
  For each $d\in P$ let $a_i'=d\cap c'$ and put
  $a_k=\bigcup_{i<|P|}a_i'$. This defines $a_k$ and if $Q$
  is empty, then there is nothing more to do. However, if
  $Q$ is nonempty, then for each $d\in Q$ we make a sequence
  of promises to construct for each $d\in Q$ a sequence
  $\langle a_{k_j}:j<\omega\rangle$ of sets which are clopen
  and such that $d\cap c'=\bigcup_j a_{k_j}$. The fact that
  $a_{k_j}$ are clopen together with Lemma \ref{openclopen}
  will guarantee that (v) is satisfied. Thus, we add to our
  list of promises all pairs $(j,c'\cap d)$ with $d\in Q$ and
  $j\in c'\cap d$.

  \textbf{Case (c)}. Suppose we are in a position to fulfill
  a promise $(j,c)$. Find an atom $d$ of $A_k$ with $j\in
  d$. Since $c$ was co-$J$ in an atom $d'$ of some $A_i$
  with $i<k$ with $j\in d'$, it must be the case that $c\cap
  d$ is co-$J$ in $d$. Use Lemma \ref{separation} to find
  two disjoint $J$-positive clopen subsets $U,V$ of $c\cap
  d$ with $j\in U$. Put $a_k=U$.

  This ends the inductive construction and the proof.
\end{proof}

\section{Descriptive complexity of ideals with topological
  representations}\label{sec:complex}

\begin{proof}[Proof of Theorem \ref{complex}]
  Suppose $J$ is analytic and has a topological
  representation. By Theorem \ref{meager} there is a
  $\sigma$-ideal $I$ of compact subsets of a Polish space
  $X$ with a countable dense set $D$ such that $J_I$ is
  isomorphic to $J$. Consider the function $a\mapsto\cl(a)$
  from $P(D)$ to $K(X)$ and note that it is Baire class
  1. Since $J_I$ is analytic, $I=\left\{A\in K(X):\exists
    b\in J_I\ A\subseteq\cl(b)\right\}$ is also analytic,
  and hence $\mathbf{\Pi}^0_2$ by the theorem of
  Kechris--Louveau--Woodin \cite[Theorem
  11]{klw}. Therefore, $J$ must be $\mathbf{\Pi}^0_3$ as a
  preimage of a $\mathbf{\Pi}^0_2$ set by a Baire class 1
  function.

  To check that $J$ is in fact $\mathbf{\Pi}^0_3$-complete,
  we need the following standard fact.

  \begin{lemma}\label{LemmaFsigmaHard}
    All analytic ideals are $\mathbf{\Sigma}^0_2$-hard.
  \end{lemma}

  \begin{proof}
    This follows directly from the
    Jalali-Naini--Mathias--Talagrand theorem \cite[Theorem
    4.1.2]{bartoszynski}. Indeed, if $J$ is analytic, then
    it has the Baire property and hence
    $\Fin\leq_{\mathrm{RB}}J$. From this we easily get a
    continuous reduction from $\Fin\subseteq2^\omega$ (which
    is $\mathbf{\Sigma}^0_2$-complete) to $J$.
  \end{proof}

  We will now show that $J_I$ is
  $\mathbf{\Pi}^0_3$-hard. The argument is based on ideas
  from \cite{msz}. Fix a point $x\in X$. Fix also a
  compatible metric on $X$ and let $V_n=\ball(2^{-n},x)$ for
  each $n\in\omega$. For each $n$ define an ideal
  $J^n=\left\{a\cap V_n: a\in J_I\right\}$ on $D\cap
  V_n$. Note that each $J^n$ is analytic, and hence
  $\mathbf{\Sigma}^0_2$-hard by Lemma
  \ref{LemmaFsigmaHard}. Therefore for each $n$ there is
  $\phi_n:2^\omega\rightarrow P(D\cap V_n)$ such that
  $\phi^{-1}_n(J^n)=\Fin$. Define
  $\phi:\left(P(\omega)\right)^\omega\rightarrow P(D)$ by
  $\phi(\langle a_n:
  n\in\omega\rangle)=\bigcup_{n\in\omega}\phi_n(a_n)$.  Let
  $W$ be the set
  $$\{\langle a_n:n\in\omega\rangle:\forall n\in\omega\ 
  a_n\in\Fin\}=\{\langle a_n: n\in\omega\rangle:\forall
  n\in\omega\ \phi_n(a_n)\in J^n\}.$$ and note that $W$ is
  $\mathbf{\Pi}^0_3$-complete.  To finish the proof it
  suffices to show that $\phi^{-1}(J_I)=W$.

  Suppose first that $\phi(\langle a_n:n\in\omega\rangle)\in
  J_I$, i.e $\bigcup_{n\in\omega}\phi_n(a_n)\in J_I$. Then
  for each $n$ we have $\phi_n(a_n)\in J_I$ and
  $\phi_n(a_n)\subseteq V_n$. Hence, for each $n$ the set
  $\phi_n(a_n)$ is in $J^n$ and $\langle
  a_n:n\in\omega\rangle\in W$.

  On the other hand, if $\langle a_n:n\in\omega\rangle\in
  W$, i.e. for each $n$ we have $\phi_n(a_n)\in J^n\subseteq
  J_I$, then $\cl(\phi_n(a_n))\in I$. Since $I$ is a
  $\sigma$-ideal,
  $\bigcup_{n\in\omega}\cl(\phi_n(a_n))\cup\left\{x\right\}$
  is in $I$. To prove that
  $\bigcup_{n\in\omega}\phi_n(a_n)\in J_I$ it suffices to
  show that $\cl(\bigcup_{n\in\omega}\phi_n(a_n))\subseteq
  \bigcup_{n\in\omega}\cl(\phi_n(a_n))\cup\left\{x\right\}$. Indeed,
  if $y\in\cl(\bigcup_{n\in\omega}\phi_n(a_n))$, then there
  is a sequence $x_n$ of elements of
  $\bigcup_{n\in\omega}\phi_n(a_n)$ convergent to $y$. There
  are two cases:

  \begin{itemize}
  \item either there is $m$ such that there are infinitely
    many $x_n$'s in $\phi_m(a_m)$. In this case $y$ is an
    element of $\cl(\phi_m(a_m))$;
  \item or in each of $\phi_m(a_m)$'s there are only
    finitely many $x_n$'s. But then $x_n$'s must converge to
    $x$ and hence $y=x$.
  \end{itemize}
  In both cases we have $y\in
  \bigcup_{n\in\omega}\cl(\phi_n(a_n))\cup\left\{x\right\}$,
  which ends the proof.
\end{proof}

Corollary \ref{klw} is now an adaptation of the proof of the
Kechris--Louveau--Woodin theorem.

\begin{proof}[Proof of Corollary \ref{klw}]
  Suppose $J=J_I$ for a $\sigma$-ideal $I$ on a compact
  space $X$ with a dense subset $D\subseteq X$ on which $J$
  lives. Consider the family $F=\{C\in K(X): C\cap D\mbox{
    is dense in }C\}$ and note that $F$ is Borel and the map
  $C\mapsto C\cap D$ is a Borel function from $F$ to
  $P(D)$. Let $I'=I\cap F$ and note that $I'$ is coanalytic
  as it is the preimage of $J$ by the above function.

  If $I'$ is $\gdelta$, then $J$ is $\fsigmadelta$-complete
  by Theorem \ref{complex}. On the other hand, if $I'$ is
  not $\gdelta$, then by the Hurewicz separation theorem
  \cite[Theorem 21.18]{kechris}, there is a Cantor set
  $C\subseteq K(X)$ such that $C\cap I'=\Q$. Now consider
  the function $K\mapsto(\bigcup K)\cap D$ from $K(C)$ to
  $P(D)$. Since $\bigcup K$ is compact for a compact
  $K\subseteq K(C)$, we have that if $K\subseteq\Q$, then
  $\bigcup K\in I'$ and hence $(\bigcup K)\cap D\in J_I$. On
  the other hand, if $K\not\subseteq \Q$, then $\bigcup K$
  does not belong to $I$ but still belongs to $F$, and hence
  $(\bigcup K)\cap D\notin J$. This proves that the above
  function is a reduction from $K(\Q)$ (which is a
  $\coana$-complete subset of $K(\cantor)$) to $J$ and shows
  $\coana$-completeness of $J$.
\end{proof}

\section{Weakly selective ideals}\label{sec:ws}

We follow standard set-theoretic notation concerning
trees. In particular, a \textit{branch} through a tree
$T\subseteq\omega^{<\omega}$ is a sequence $t\in\baire$ such
that $t\restriction n\in T$ for every $n\in\omega$. The set
of all branches through a tree $T$ is denoted by
$[T]$. Given a tree $T\subseteq\omega^{<\omega}$, we say
that branch $(n_1,n_2,\ldots)\in[T]$ is
\textit{$J$-positive} if $\{n_1,n_2,\ldots\}\notin J$. Given
a tree $T\subseteq\omega^{<\omega}$ and $t\in T$ we write
$\mathrm{split}_T(t)=\{n\in\omega:t^\smallfrown n\in
T\}$. Given a family $A$ of subsets of $\omega$, we say that
a tree $T\subseteq\omega^{<\omega}$ is
\textit{$A$-splitting} if for each $t\in T$ we have
$\mathrm{split}_T(t)\in A$. Given a tree $T$, we call the
sets $\mathrm{split}_T(t)$ for $t\in T$ the
\textit{splitting sets of $T$}.

A subclass of weakly selective ideals are the selective
ideals (see Mathias \cite{mathias}, Farah
\cite{farah.semiselective} and Grigorieff
\cite{grigorieff}). An ideal $J$ is \textit{selective}
\cite[Definition 1.7]{grigorieff} if every $J$-partition of
$\omega$ admits a $J$-positive selector. Here, a
\textit{$J$-partition} of $\omega$ is a partition
$\omega=\bigcup_n a_n$ with $\bigcup_{m>n}a_m\notin J$ for
every $m$. Equivalently \cite[Definition
1.1]{farah.semiselective}, $J$ is selective if any sequence
of $J$-positive sets $a_n$ has a $J$-positive
\textit{diagonalization}, i.e. a set $a\notin J$ such that
$a\setminus n\subseteq a_n$ for each $n$. Selective ideals
have been studied by Grigorieff, who proved \cite[Corollary
1.15]{grigorieff} a characterization of selectivity in terms
of branches of trees: an ideal $J$ is selective if and only
if every tree $T$ with the property that any finite
intersection of splitting sets of $T$ is in $J^+$, has a
$J$-positive branch. The following lemma provides a similar
characterization of weak selectivity. The proof is similar
to that of Grigorieff and one implication is implicit in
\cite{meza-alcantara}.

\begin{lemma}\label{grigorieff}
  Let $J$ be an ideal on $\omega$. The following are
  equivalent:
  \begin{itemize}
  \item[(a)] $J$ is weakly selective,
  \item[(b)] for each $J$-positive $b$, every
    co-$(J\restriction b)$-splitting tree has a $J$-positive
    branch.
  \end{itemize}
\end{lemma}
\begin{proof}
  (b)$\Rightarrow$(a). Suppose $J$ is not weakly selective,
  i.e. there is a $J$-positive $b$ and a function
  $f:b\rightarrow\omega$ which is neither constant nor
  one-to-one on any $J$-positive subset of $b$. This means
  that preimages of single points belong to $J$.  We will
  produce a co-$(J\restriction b)$-splitting tree whose all
  branches are in $J$. Define the tree $T\subseteq
  b^{<\omega}$ inductively as follows. If
  $t=(n_0,\ldots,n_k)\in b^{<\omega}$ belongs to $T$, then
  add to $T$ all $t^\smallfrown n$ with
  $n\notin\bigcup_{i\leq k}f^{-1}(n_i)$. Note that this is a
  co-$(J\restriction b)$-splitting tree and if
  $x=(n_0,n_1,\ldots)$ is a branch through $T$, then $f$ is
  one-to-one on $\{n_0,n_1,\ldots\}$. Thus, all branches of
  $T$ are in $J$.

  (a)$\Rightarrow$(b). Suppose now $J$ is weakly selective,
  $b$ is $J$-positive and $T$ is a co-$(J\restriction
  b)$-splitting tree. We need to produce a $J$-positive
  branch through $T$. First, note that if $J$ is weakly
  selective, then any sequence $a_n$ of sets in
  co-$(J\restriction b)$ has a $J$-positive
  pseudointersection. Indeed, if $\bigcap_n a_n\notin J$,
  then this intersection is in particular a
  pseudointersection. Otherwise the sets
  $\bigcap_{m<n}a_m\setminus\bigcup_{m\geq n}a_m$ and
  $\bigcap_n a_n$ define a partition of $b$ into sets in $J$
  and any selector of that partition is a pseudointersection
  of $a_n$'s.

  Now, let $c$ be a $J$-positive pseudointersection of the
  splitting sets of $T$. Define a sequence $k_n$ of elements
  of $c$ by induction in the following way. Let $k_0$ be the
  minimum of $c$. Let $k_{n+1}$ be the minimal element of
  $c$ bigger than $k_n$ such that $c\setminus k_{n+1}$ is
  contained in $\mathrm{split}_T(t)$ for each $t\in T$ whose
  length is not greater than $k_n$ and whose all elements
  are not greater than $k_n$. The intervals $[k_n,k_{n+1})$
  define a partition of $c$ into finite sets. Let
  $d\subseteq c$ be a $J$-positive selector of that
  partition. Write $d_0=d\cap\bigcup_n[k_{2n},k_{2n+1})$ and
  $d_1=d\cap\bigcup_n[k_{2n+1},k_{2n+2})$. Note that by the
  definition of $k_n$'s, both sets $d_0$ and $d_1$ are
  branches through $T$. Since $d=d_0\cup d_1$, one of these
  sets must be $J$-positive and hence $T$ has a $J$-positive
  branch, as needed.
\end{proof}

The proof of Theorem \ref{ws} will be based on the above
lemma as well as on some ideas of Hru\v s\'ak from his
Category Dichotomy \cite[Theorem 5.20]{hrusak}. In
particular, we will use a game $H(J)$, invented by Hru\v
s\'ak in \cite{hrusak}.

\begin{proof}[Proof of Theorem \ref{ws}]
  (ii)$\Rightarrow$(i) Suppose $J=\bigcap_{l\in\Lambda}J_l$,
  where each $J_l$ has a topological representation. Let
  $f:b\rightarrow\omega$ be a function with $b\notin
  J$. Suppose $f$ is not constant on any $J$-positive subset
  of $b$. Pick $l\in\Lambda$ such that $b\notin J_l$. We
  will find a $J_l$-positive subset of $b$ on which $f$ is
  1-1. Let $X$ be a separable metric space, $D$ its dense
  countable subset and $I$ a $\sigma$-ideal of subsets of
  $X$ such that $J_l$ is isomorphic to $J_I$. Without loss
  of generality assume $b\subseteq D$. Write $B=\cl(b)$ and
  let $B'=B\setminus\{U: U\mbox{ is basic open and }B\cap
  U\in I\}$. Note that $B'$ is still $I$-positive. Enumerate
  all basic open sets in $X$ which intersect $B'$ into a
  sequence $\langle V_n:n<\omega\rangle$ and by induction on
  $i$ construct a sequence of points $n_i\in b$ such that
  $f(n_i)\notin\{f(n_j):j<i\}$ and $n_i\in V_i$. Once this
  is done put $b'=\{n_i:i<\omega\}$ and note that $b'\notin
  J_l$ since $\cl(b')$ contains $B'$. Obviously, then $f$ is
  1-1 on $b'$.

  To perform the construction, suppose that points $n_i$ are
  chosen for $i<k$ and let $a=b\cap V_k$. Note that $a\notin
  J_l$. By our assumption, $f$ is not constant on any
  $J_l$-positive set, so it cannot assume finitely many
  values on $a$, which implies that there is $n_k\in a$ such
  that $f(n_k)\not=f(n_i)$ for all $i<k$. This ends the
  construction.

  (i)$\Rightarrow$(ii) Suppose now that $J$ is coanalytic
  ideal of subsets of $\omega$. Let
  $D\subseteq[\omega]^\omega\times\baire$ be closed set such
  that $[\omega]^\omega\setminus J$ is the projection of
  $D$.

  Consider the following game $H'(J)$. In his $n$-th turn
  Player II picks a set $a_n\in J$. Player I responds with a
  pair $(k_n,m_n)$ with $k_n\in\omega\setminus a_n$ and
  $m_n\in n\cup\{$pass$\}$. Player I wins if at the end he
  has chosen infinitely many $m_n$'s different than 'pass'
  and $(\bar k,\bar m)$ belongs to $D$, where $\bar
  k=\{k_n:n<\omega\}$ and $\bar m$ is the sequence of those
  $m_n$'s which are not equal to 'pass'.

  The game $H'(J)$ is an unfolded version of the game $H(J)$
  in which Player II picks $a_n\in J$ and Player I responds
  just with a number $k_n\notin a_n$. Player I wins in $H(J)$
  if $\{k_n:n<\omega\}$ does not belong to $J$.

  \begin{claim}\label{strategy1}
    If Player II has a winning strategy in $H'(J)$, then he
    also has a winning strategy in $H(J)$.
  \end{claim}
  \begin{proof}
    Let $\sigma$ be a winning strategy for Player II in
    $H'(J)$. We describe a strategy $\sigma'$ for Player II
    in the game $H(J)$. Suppose Player II is about to make
    his $n$-th move after Player I has played
    $k_0,\ldots,k_{n-1}$. Let $F$ be the finite set of all
    sequences $m_0,\ldots,m_n$ such that $m_i\in
    i\cup\{$pass$\}$ and for each $(m_0,\ldots,m_{n-1})\in
    F$ let $a_f$ be the $n$-th move in the game $H'(J)$
    according to the strategy $\sigma$ after Player I has
    played $(k_0,m_0),\ldots,(k_{n-1},m_{n-1})$. Let the
    move of Player II in $H(J)$ be $\bigcup_{f\in F}a_f$.

    We claim that this is a winning strategy for Player
    II. Suppose it is not and there is a counterplay of
    Player I. The counterplay is a sequence $(k_n:n<\omega)$
    such that $\{k_n:n<\omega\}\notin J$. We will find a
    counterplay to the strategy $\sigma$ in $H'(J)$. Since
    $\{k_n:n<\omega\}\notin J$, there is a sequence
    $(m_n:n<\omega)\in\baire$ such that
    $(\{k_n:n<\omega\},(m_n:n<\omega))\in D$. Let $m_n'$ be
    the sequence such that $m_n\in i\cup\{$pass$\}$ and the
    elements of $m_n'$ different from 'pass' enumerate
    $(m_n:n<\omega)$. Consider now the play in which Player
    I plays $(k_n,m_n')$ and Player II plays according to
    $\sigma$. Note that this is a legal play in $H'(J)$ by
    the definition of $\sigma'$. It is also a counterplay to
    $\sigma$ in which Player I wins.
  \end{proof}

  \begin{claim}\label{strategy2}
    If $J$ is weakly selective, then Player II cannot have a
    winning strategy in $H(J)$.
  \end{claim}
  \begin{proof}
    Suppose there is such strategy and let $T$ be the tree
    of all counterplays of Player I,
    i.e. $T=\{(k_0,\ldots,k_n): n<\omega$ and
    $k_i\notin\sigma(k_0,\ldots,k_{i-1})$ for each $i\leq
    k\}$. Note that $T$ is a co-$J$-splitting tree whose all
    branches belong to $J$. But by Lemma \ref{grigorieff},
    any co-$J$-splitting tree must have a $J$-positive
    branch.
  \end{proof}
   
  \begin{claim}\label{family}
    If $J$ is weakly selective, then for every $b\notin J$
    there is a countable family $Y$ of $J$-positive subsets
    of $b$ such that for every $a\in J$ there is $x\in Y$
    with $x\cap a=\emptyset$.
  \end{claim}
  \begin{proof}
    Given a set $b\notin J$ consider the ideal
    $J\restriction b=\{a\subseteq b: a\in J\}$ of subsets of
    $b$. Note that it is still weakly selective and
    coanalytic. Hence, by Claims \ref{strategy1} and
    \ref{strategy2} and the fact that the game
    $H'(J\restriction b)$ is a closed game, there is a
    winning strategy $\sigma$ for Player I in
    $H'(J\restriction b)$. Let $T$ be the tree of all
    partial plays in $H'(J\restriction b)$, i.e. all
    sequences $(a_0,n_0,m_0,a_1,n_1,m_1,\ldots,a_k,n_k,m_k)$
    such that
    $(n_i,m_i)=\sigma((a_0,n_0,m_0,\ldots,a_i))$. Now,
    inductively, for each $k$ find a subset $T_k'$ of $T$
    consisting of sequences
    $(a_0,n_0,m_0,\ldots,a_k,n_k,m_k)$ such that:
    \begin{itemize}
    \item $(a_0,n_0,m_0,\ldots,a_{k-1},n_{k-1},m_{k-1})$
      belongs to $T_{k-1}'$,
    \item if $(a_0,n_0,m_0,\ldots,a_k,n_k,m_k)$ and
      $(a_0',n_0,m_0,\ldots,a_k',n_k,m_k)$ belong to $T_k'$, then
      $a_i=a_i'$ for each $i\leq k$.
    \end{itemize}
    Let $T'=\bigcup_k T_k'$ and put
    \begin{align*}
      S=\{(n_0,\ldots,n_k)\in\omega^{<\omega}:\exists(a_0,\ldots,a_k),(m_0,\ldots,m_k)\\
      (a_0,n_0,m_0,\ldots,a_k,n_k,m_k)\in T'\}.
    \end{align*}
    The tree $S$ is a subtree of $b^{<\omega}$, whose all
    branches are $(J\restriction b)$-positive, as if
    $(n_0,n_1,\ldots)\in[S]$, then there are
    $(a_0,a_1,\ldots),(m_0,m_1,\ldots)$ such that
    $(a_0,n_0,m_0,a_1,n_1,m_1,\ldots)\in[T']$ and all
    branches through $T'$ follow the strategy $\sigma$. The
    tree $S$ is also $(J\restriction
    b)^+$-splitting. Indeed, if $a=\mathrm{split}_S(t)\in J$
    for some $t=(n_0,\ldots,n_k)\in S$, then pick
    $(a_0,\ldots,a_k)$ and $(m_0,\ldots,m_k)$ such that
    $(a_0,n_0,m_0,\ldots,a_k,n_k,m_k)\in T_k'$. Let then
    $$(n,m)=\sigma((a_0,n_0,m_0\ldots,a_k,n_k,m_k,a)).$$ Note
    that $n\notin a$ and by the construction, there is $a'$
    such that $$(a_0,n_0,m_0\ldots,a_k,n_k,m_k,a',n,m)\in
    T_{k+1}'.$$ So $(n_0,\ldots,n_k,n)\in S$ and this
    contradicts the fact that $n\notin a$.

    Consider now the family of all splitting sets of $S$. We
    claim that this is the desired family. Indeed, if $a$
    intersects all these sets, then $a$ contains a branch
    through $S$ and therefore, it is $J$-positive.
  \end{proof}
  
  For each $b\notin J$ let $Y_b$ be a countable family of
  subsets of $b$ as in Claim \ref{family}. We say that a
  family $X$ of $J$-positive sets is \textit{almost
    separating} if
  \begin{itemize}
  \item for every $x\in X$ and $n\in x$ there is $y\in X$
    with $y\subseteq x$ and $n\notin y$,
  \item for every $x\in X$ there are $x_0,x_1\in X$ with
    $x_0,x_1\subseteq x$ such that $x_0\cap x_1=\emptyset$ and
    $x\setminus(x_0\cup x_1)\in J$.
  \end{itemize}
  We also say that a family $X$ of $J$-positive sets is
  \textit{almost closed under finite intersections} if
  \begin{itemize}
  \item every finite intersection of elements of $X$ is
    either empty or $J$-positive,
  \item for every $x_0,\ldots,x_n\in X$ if $\bigcap_{i\leq
      n}x_i\notin J$, then there is $y\in X$ such that
    $y\subseteq \bigcap_{i\leq n}x_i$ and $(\bigcap_{i\leq
      n}x_i)\setminus y\in J$.
  \end{itemize}
  \begin{claim}\label{final}
    For each $b\notin J$ there is a countable family $X$ of
    $J$-positive subsets of $b$ which is almost closed under
    finite intersections, almost separating and such that
    for each $x\in X$ and $a\in J$ there is $y\in X$ with
    $y\subseteq x$ and $y\cap a=\emptyset$.
  \end{claim}
  \begin{proof}
    First note that since $J\restriction x$ is not a maximal
    ideal (as it has the Baire property) for every
    $J$-positive set $x\subseteq b$, there are two disjoint
    complementary $J$-positive subsets of $x$, say
    $x(0),x(1)$. Moreover, there is a countable family $S_x$
    of $J$-positive subsets of $x$ which separates points in
    $x$. Let $Z$ be a countable family of $J$-positive
    subsets of $b$ such that $b\in Z$ and
    \begin{itemize}
    \item for every $c\in Z$ we have $Y_c\subseteq Z$,
    \item for every $c\in Z$ we have $c(0),c(1)\in Z$ and
      $S_c\subseteq Z$,
    \item for every $c_0,\ldots,c_n\in Z$ if $\bigcap_{i\leq
        n}c_i\notin J$, then $\bigcap_{i\leq n}c_i\in Z$.
    \end{itemize}
    Enumerate $Z$ with infinite repetitions as
    $\{z_0,z_1,\ldots\}$. Now, by induction construct sets
    $x_i$ as follows. Let $x_0=z_0$ and
    $x_{i+1}=z_i\setminus\bigcup\{z_i\cap x_j: j<i\mbox{ and
    }z_i\cap x_j\in J\}$. Now, the family
    $X=\{x_i:i<\omega\}$ is as needed. Indeed, note that for
    each $i<\omega$ we have $x_i\subseteq z_i$ and
    $z_i\setminus x_i\in J$, so the properties of the family
    $X$ follow immediately from the construction of the
    family $Z$.
  \end{proof}

  Given a set $b\notin J$ let $X_b$ be a countable family of
  $J$-positive subsets of $b$ as in Claim
  \ref{final}. Let $$J_b=\{a\subseteq\omega:\forall x\in
  X_b\ \exists y\in X_b\quad y\subseteq x\ \wedge\ a\cap
  y=\emptyset\}.$$ Note that $J_b$ is an ideal of subsets of
  $\omega$.

  \begin{claim}
    For each $b\notin J$ the ideal $J$ is contained in
    $J_b$.
  \end{claim}
  \begin{proof}
    This follows directly from the properties of $X_b$.
  \end{proof}

  \begin{claim}
    For each $b\notin J$ the ideal $J_b$ has a topological
    representation.
  \end{claim}
  \begin{proof}
    We will check that $J_b$ is countably separated and
    dense.
     
    $J_b$ is countably separated by $X_b$. Indeed, Let $a\in
    J_b$ and $c\notin J_b$. Since $c\notin J_b$, there is
    $x\in X_b$ such that for no $y\in X_b$ with $y\subseteq
    x$ it is the case that $y\cap c=\emptyset$. Note that
    actually for each such $y$ we have $y\cap c\notin
    J_b$. Now, since $a\in J_b$, there is $y\in X_b$
    with $y\subseteq x$ and $y\cap a=\emptyset$.

    To see that $J_b$ is dense, let $c\subseteq\omega$ be
    infinite. We need to find an infinite $a\subseteq c$
    such that $a\in J_b$. We can assume that $c\notin
    J_b$. This means that there is $x\in X_b$ such that for
    each $y$ with $y\subseteq x$ we have $y\cap c\notin
    J_b$. Enumerate $X_b$ as $\{x_i:i<\omega\}$. By
    induction on $i$, construct a strictly increasing
    sequence $n_i$ and $J$-positive sets $y_i\in X_b$ such
    that
    \begin{itemize}
    \item[(a)] $y_{i+1}\subseteq y_i$, $y_i\subseteq x$ (so
      $y_i\cap c\notin J_b$) and $n_i\in y_i\cap c$
    \item[(b)] $x_i\setminus y_i$ contains an element of
      $X_b$.
    \end{itemize}
    We start with $y_{-1}=x$. To perform the induction step,
    use the fact that $X_b$ is almost separating and find
    two $J_b$-positive sets $x_i(0),x_i(1)\in X_b$ which are
    subsets of $x_i$ and such that $x_i\setminus(x_i(0)\cup
    x_i(1))\in J$. If $y_{i-1}\cap x_i$ is empty, then put
    $y_i=y_{i-1}$ and pick any $n_i\in y_i\cap c$ bigger
    than $n_{i-1}$. If $y_{i-1}\cap x_i$ is nonempty, then
    it is $J$-positive. Note that at least one of
    $y_{i-1}\cap x_i(0)$ or $y_{i-1}\cap x_i(1)$ must be
    nonempty. Since $X_b$ is almost closed under finite
    intersections, one of these sets contains an element of
    $X_b$, say $y_i$. Pick any $n_i\in y_i\cap c$ bigger
    than $n_{i-1}$. This ends the construction.

    Put $a=\{n_i:i<\omega\}$. We claim that $a$ belongs to
    $J_b$. Indeed, pick $x_i\in X_b$. By (b) $x_i\setminus
    y_i$ contains an element of $X_b$, say $y$. Hence, by
    (a), $a\cap y$ is finite. Since $X_b$ almost separates
    points, we can further shrink $y$ to $z\in X_b$ such
    that $z\cap a$ is empty. This shows that $a\in J_b$.
  \end{proof}
  Now, for each $b\notin J$ we have the ideal $J_b$ such
  that $J\subseteq J_b$, $b\notin J_b$ and $J_b$ has a
  topological representation. This implies that
  $J=\bigcap_{b\notin J}J_b$ is an intersection of ideals
  which have topological representations and ends the proof.
\end{proof}

\bibliographystyle{plain}
\bibliography{refs}

\end{document}